\documentclass[11pt]{amsart}
\usepackage{mathrsfs}
\usepackage{amscd}
\usepackage{amsmath}
\usepackage{graphicx}
\usepackage{amsfonts}
\usepackage{amssymb}
\usepackage{color}
\usepackage{hyperref}
\newtheorem{set}{set}[section]
\newtheorem{Corollary}[set]{Corollary}

\newtheorem{Proposition}[set]{Proposition}

\newtheorem{Theorem}[set]{Theorem}
\begin{document}

\title[Algebraic Maximal Numerical Ranges on $C^*$-Algebras]
{Algebraic Maximal Numerical Range and its preservers of Triple Products on $C^*$-Algebras}

\author{Yongqi Fan, Xiaofei Qi, Jinchuan Hou}


\begin{abstract}

Let $\mathcal{A}$ and $\mathcal{B}$ be unital $C^*$-algebras, and let $V_0(a)=\{f(a): f\in\mathcal S(\mathcal A), f(a^*a)=\|a\|^2\}$ be the algebraic maximal numerical range of  $a\in\mathcal{A}$, where $\mathcal S(\mathcal A)$ is the set of all states of $\mathcal A$.
We study the properties of $V_0(a)$ and characterize surjective maps preserving $V_0$ of
triple products. We show  that if $\Phi\colon\mathcal{A}\to\mathcal{B}$ satisfies
\(
V_0(\Phi(a)\Phi(b)\Phi(c))=V_0(abc) \text{~for all~} a,b,c\in\mathcal{A},
\)
then the map $a\mapsto \Phi(1_{\mathcal{A}})^{-1}\Phi(a)$ is a
multiplicative bijection. Furthermore, for von Neumann algebras
without central summands of type $I_1$ or prime $C^*$-algebras of real rank zero, such preservers are precisely $*$-isomorphisms multiplied by a central element $u\in Z(\mathcal{B})$ with $u^3=1$.

\end{abstract}

\maketitle

\section{Introduction}

Let \(\mathcal{A}\) be a unital \(C^*\)-algebra with unit \(1_{\mathcal{A}}\), and let $\mathcal{S}(\mathcal{A})$ denote its state space
\[
\mathcal{S}(\mathcal{A})=\{f\in \mathcal{A}^*: f(1_\mathcal{A})=\|f\|=1\}.
\]
Here \(\mathcal{A}^*\) denotes
the dual space of $\mathcal A$.
Equivalently, \(\mathcal{S}(\mathcal{A})\) is the set of all positive linear functionals \(f\) on
\(\mathcal{A}\) satisfying \(f(1_\mathcal{A})=1\).
For \(a\in \mathcal{A}\), the algebraic numerical range and numerical radius of
\(a\) are defined respectively by
\[
V(a)=\{f(a): f\in \mathcal{S}(\mathcal{A})\}\text{~and~}v(a)=\sup\{|\lambda|:\lambda\in V(a)\}.
\]
It is well known that \(V(a)\) is a compact convex subset of $\mathbb{C}$ and contains the convex hull of the spectrum of $a$, see \cite{stampfli1968growth}.
Furthermore, the numerical radius \(v(\cdot)\) defines a norm on \(\mathcal{A}\), which is
equivalent to the C$^*$-norm. More precisely,
\[
\frac{1}{2}\|a\|\leq v(a)\leq \|a\| \text{~for all~} a\in\mathcal{A}.
\]

Let $\mathcal{B}(H)$ be the $C^*$-algebra of all bounded linear operators acting on a complex Hilbert space $H$ with inner
product $\langle \cdot, \cdot \rangle$.
Recall that, for $T \in \mathcal{B}(H)$, the classical numerical range and numerical radius are defined by \cite{gau2021numerical}
\[
W(T) = \{ \langle Tx, x \rangle : x \in H, \|x\| = 1 \}\text{~and~}w(T)=\sup\{|\lambda|:\lambda\in W(T)\}.
\]
The Toeplitz-Hausdorff theorem asserts that $W(T)$ is a convex subset of $\mathbb{C}$. Moreover, $\overline{W(T)} = V(T)$, where $\overline{M}$ is the closure of a subset $M$ of $\mathbb{C}$.
The numerical range of an operator is a fundamental concept that has been widely studied in both theoretical and applied settings. In particular, numerical range preserving maps on various operator algebras
have attracted considerable attention; see \cite{li1987linear}, \cite{chan1995numerical} and \cite{cui2003linear}.
The results on preserver problems of numerical range reveal that the numerical range contains enough information so that it has strong rigidity that determine the $*$-isomorphisms between $C^*$-algebras.

For a fixed operator $T\in \mathcal{B}(H)$, the map $\delta _T(X)=TX-XT$ defines an inner derivation induced by $T$.
To determine the norm of an inner derivation on $\mathcal{B}(H)$, the notion of the maximal numerical range was introduced and studied by Stampfli in 1970, see \cite{stampfli1970norm}.
Let $T \in \mathcal{B}(H)$. The maximal numerical range of $T$ is defined by
\[
W_0(T) = \{ \lim_{n\to\infty} \langle Tx_n, x_n \rangle : x_n \in H, \|x_n\| = 1, \lim_{n\to\infty} \|Tx_n\| = \|T\| \}.
\]
If $H$ is finite-dimensional, then $W_0(T)$ coincides with the numerical range generated by the maximal vectors, that is,
\[
W_0(T) = \{ \langle Tx, x \rangle : x \in H, \|x\| = 1, \|Tx\| = \|T\| \}.
\]
Obviously, the set $W_0(T)$ is a nonempty, convex, and compact subset of the closure of the numerical range $W(T)$, that is, $W_0(T)\subseteq \overline{W(T)}$ (see \cite{stampfli1970norm}). Define
\[
w_0(T) = \sup \{ |\lambda|:\lambda \in W_0(T) \}.
\]
Note that $w_0(T)$ is not a norm, since it fails to be positive definite. For example, let
\[ T = \begin{pmatrix} 0 & 1 \\ 0 & 0 \end{pmatrix} \]
Then \( \|T\| = 1 \), \( W_0(T) = \{0\} \), and \( w_0(T) = 0 \), while \( T \neq 0\).

The corresponding notion for a unital C$^*$-algebra is formulated in
terms of maximal states. Let
\[
\mathcal{S}_{\max}(a)=\{f\in \mathcal{S}(\mathcal{A}):f(a^*a)=\|a\|^2\}
\]
and define the algebraic maximal numerical range and maximal numerical radius of \(a\) by
\[
V_0(a)=\{f(a):f\in \mathcal{S}_{\max}(a)\}\text{~and~}v_0(a)=\sup\{|\lambda|:\lambda\in V_0(a)\}.
\]
The set \(V_0(a)\) is a nonempty compact convex subset of \(V(a)\). Moreover, in the special case when \(\mathcal{A}=\mathcal{B}(H)\), this algebraic definition agrees with the maximal
numerical range introduced by Stampfli through vector sequences, that is  $V_0(T)=W_0(T)$.

The numerical range is a fundamental concept. As an important related notion, has become a useful tool connecting numerical range theory with spectral and geometric aspects of functional analysis; see \cite{ji2007essential}, \cite{amara2026maximal}, \cite{runji2017maximal}.
The maximal numerical range has been investigated for quadratic elements, hyponormal operators, operator tensor products, and special matrix classes, respectively, in \cite{benabdi2021maximal}, \cite{baghdad2019maximal}, \cite{taki2025maximal}, \cite{hamed2018maximal}. Its counterpart in the framework of $C^*$-algebras has also been thoroughly studied; see, for example, \cite{benabdi2021maximal}, \cite{benabdi2022birkhoff}.

Preserver problems for maximal numerical ranges have recently attracted considerable attention, see \cite{dhifaoui2024maps}, \cite{bourhim2025multiplicatively} and \cite{bourhim2025maps}. In \cite{dhifaoui2024maps}, K. Dhifaoui and M. Mabrouk characterized the structure of all surjective maps on $\mathcal{B}(H)$ that preserve the maximal numerical range of products $TS$ and skew products $TS^*$ of operators. They showed that there exists a unitary operator $U \in \mathcal{B}(H)$ such that $\Phi(T) = \pm UTU^*$ for all $T \in \mathcal{B}(H)$.
In \cite{bourhim2025multiplicatively}, A. Bourhim and M. Mabrouk characterized all surjective maps $\Phi_1$ and $\Phi_2$ on $\mathcal{B}(H)$ such that, for all $T$, $S\in \mathcal{B}(H)$, the maximal numerical ranges of $TS$ and $\Phi_1(T)\Phi_2(S)$  coincide if and only if there exist an invertible operator $A \in \mathcal{B}(H)$ and a unitary operator $U\in \mathcal{B}(H)$ such that $\Phi_1(T) = UTA$ for all $T \in \mathcal{B}(H)$ and $\Phi_2(T) = A^{-1}TU^*$ for all $T \in \mathcal{B}(H)$.
In \cite{bourhim2025maps}, A. Bourhim and M. Mabrouk obtained a complete characterization of all surjective maps on $\mathcal{B}(H)$ preserving the maximal numerical range of the triple product $TST$. More precisely, such maps $\Phi$ have the form that there is a scalar $\lambda \in \mathbb{C}$ with $\lambda^3 = 1$ and a unitary operator $U \in \mathcal{B}(H)$ such that either $\Phi(T) = \lambda UTU^*$ for all $T \in \mathcal{B}(H)$ or $\Phi(T) = \lambda UT^\top U^*$ for all $T \in \mathcal{B}(H)$. They also characterized maps preserving the maximal numerical range of the skew triple product $T S^*T$. In this case, such a map has the form that there is a unitary operator $U \in \mathcal{B}(H)$ such that either $\Phi(T) = UTU^*$ for all $T \in \mathcal{B}(H)$ or $\Phi(T) = UT^\top U^*$ for all $T \in \mathcal{B}(H)$.

The purpose of this paper is to study further the properties of the algebraic maximal numerical ranges  and the corresponding preserver
problems for the algebraic maximal numerical range on unital \(C^*\)-algebras.
We show that   if a surjective
map \(\Phi:\mathcal A\to\mathcal B\) between unital $C^*$-algebras  preserves the algebraic maximal numerical range of triple product \(abc\), that is,
\[
V_0(\Phi(a)\Phi(b)\Phi(c))=V_0(abc) \text{~for all~} a, b, c \in \mathcal{A}
\]
then $\Phi(1_\mathcal{A})$ is invertible and $\Psi(\cdot)=\Phi(1_\mathcal{A})^{-1}\Phi(\cdot)$ is a multiplicative bijection.
Under the standard assumption that $\mathcal A,\mathcal B$ are  von Neumann algebras without central summands of type $I_1$, or unital $C^*$-algebras of real rank zero, then $\Phi$ preserves the algebraic maximal numerical range of triple product \(abc\) if and only if there exist  a central element $u\in Z(\mathcal{B})$ with $u^3=1$ and a
$*$-isomorphism $\Psi$ such that $\Phi=u\Psi$. Though there is little information contributed from the algebraic maximal numerical range, our results reveal that, it still has certain rigidity to determine the structure of von Neumann algebras and even $C^*$-algebras.

The rest of the paper is organized as follows. In Section 2, we recall some known properties of the algebraic maximal numerical range and provide several new properties.
In Section 3, we draw the multiplicativity
 of surjective maps between unital $C^*$-algebras that preserve the algebraic maximal numerical range of triple products $abc$.
We also show that, for von Neumann
algebras without central summands of type $I_1$ and  unital C$^*$-algebras of real rank zero, the algebraic maximal numerical range of triple product is a rigid invariant for $*$-isomorphisms.

We conclude this section by fixing the notations and recalling several basic facts that will be used throughout the paper. We denote by $\mathbb{C}$ and $\mathbb{R}$ the fields of complex and real numbers, respectively. Let $\mathcal{A}$ be a unital $C^*$-algebra with unit $1_{\mathcal{A}}$. For $a\in\mathcal{A}$, we denote its spectrum and spectral radius by $\sigma(a)$ and $r(a)$, respectively.
We write $\mathcal{A}_+$ for the set of all positive elements in $\mathcal{A}$. $Z(\mathcal{A})$  and $\mathcal{P}(\mathcal{A})$ stand respectively   for the  center of $\mathcal A$ and the set of all projections in $\mathcal{A}$.
If $\mathcal{A}$ is a von Neumann algebra and  $a=a^*\in\mathcal{A}$, we denote by $E_a$ the spectral measure of $a$. For every Borel set $\Delta\subseteq\sigma(a)$, the corresponding spectral projection $E_a(\Delta)$ belongs to $\mathcal{A}$. Consequently, every self-adjoint element of $\mathcal{A}$ can be approximated in norm by self-adjoint elements with finite spectrum. \if false whose spectral projections belong to $\mathcal{A}$.\fi
Also, recall that an element $a\in\mathcal{A}$ is called hyponormal if
\(
a^*a\geq aa^*,
\) see \cite{taki2025maximal} for more details.

\section {Some properties on maximal numerical range}

In this section, we introduce some notation and collect several properties that will be used throughout the paper.

 Let $(X,d)$ be a metric space, and let $A$ and $B$ be nonempty compact
subsets of $X$. Recall that (Ref. \cite{berinde2022pompeiu})
the Hausdorff distance between $A$ and $B$ is defined by
\[
d_H(A,B)
:=
\max\left\{
\sup_{x\in A} d(x,B),
\sup_{y\in B} d(y,A)
\right\},
\]
where,
for $x\in X$ and a nonempty subset $E\subseteq X$,
$
d(x,E):=\inf_{y\in E} d(x,y).
$

It is well known that $d_H$ defines a metric on the family of
nonempty compact subsets of $X$. In particular,
\(
d_H(A,C)\leq d_H(A,B)+d_H(B,C).
\)

 Let $\mathcal{A}$ be a $C^*$-algebra and let $f\in \mathcal{S}(\mathcal{A})$.
Also recall that a net $(a_\alpha)_{\alpha\in\Lambda}\subseteq\mathcal{A}_+$ with
$\|a_\alpha\|=1$ for every $\alpha\in\Lambda$ is said to
excise $f$ \cite[Definition 2.1]{akemann1986excising} if
\[
\lim_{\alpha} \left\| a_\alpha a a_\alpha-f(a)a_\alpha^2 \right\|=0
\text{ for every }a\in\mathcal{A}.
\]

\begin{Proposition}\label{prop2.1} {\rm (\cite[Proposition 2.2]{akemann1986excising})}
Let $\mathcal{A}$ be a $C^*$-algebra. Every pure state
$f\in \mathcal{S}(\mathcal{A})$ is excised by a decreasing net
$(a_\alpha)_{\alpha\in\Lambda}\subseteq\mathcal{A}_+$ such that
\(
f(a_\alpha)=1
\)
for every $\alpha\in\Lambda$.
\end{Proposition}

\begin{Proposition}\label{prop2.2} Let \(\mathcal A\) be a unital \(C^*\)-algebra and let
$f\in \mathcal{S}(\mathcal{A})$ be a pure state.
Then there exists a decreasing net $(e_\alpha)_{\alpha\in\Lambda}\subseteq\mathcal{A}_+$ that excises $f$ and satisfies
\(
f(e_\alpha)=1
\) ($\alpha\in\Lambda$),
such that, for every $a\in\mathcal{A}$,
\[
\lim_{\alpha}
d_H\bigl(V_0(ae_\alpha),\{f(a)\}\bigr)
=0.
\]
Here, \(d_H\) denotes the Hausdorff metric on the nonempty compact subsets
of \(\mathbb C\).
\end{Proposition}

\begin{proof} By Proposition \ref{prop2.1}, there exists a decreasing net
$(e_\alpha)_{\alpha\in\Lambda}\subseteq\mathcal{A}_+$ that excises $f$. Moreover we obtain that each $\|e_\alpha\| = 1$ and
\[
\lim\limits_{\alpha} \|e_\alpha x e_\alpha - f(x) e_\alpha^2\| = 0 \text{~for every~} x \in  \mathcal{A}.
\]
In particular, $0 \le e_\alpha \le 1$ and $f(e_\alpha) = 1$.
By taking $x = a$ and $x = a^*a$ respectively, we obtain
\[
\lim_{\alpha} \|e_\alpha a e_\alpha - f(a) e_\alpha^2\| = 0
\]
and
\[
\lim_{\alpha} \|e_\alpha a^*a e_\alpha - f(a^*a) e_\alpha^2\| = 0.
\]
Since $V_0(ae_\alpha)$ is a nonempty compact subset of $\mathbb{C}$, the Hausdorff distance between $V_0(ae_\alpha)$ and the singleton $\{f(a)\}$ is well defined and is given by
\begin{align*}
d_H(V_0(ae_\alpha), \{f(a)\})
&= \max\left\{\sup_{\lambda_\alpha \in V_0(ae_\alpha)} d(\lambda_\alpha, \{f(a)\}), d(f(a), V_0(ae_\alpha))
\right\} \\
&= \max\left\{\sup_{\lambda_\alpha \in V_0(ae_\alpha)} |\lambda_\alpha - f(a)|, \inf_{\mu_\alpha  \in V_0(ae_\alpha)} |f(a) - \mu_\alpha |  \right\}.
\end{align*}
Observe that
\[
\inf_{\mu_\alpha  \in V_0(ae_\alpha)} |f(a) - \mu_\alpha |
\leq \sup_{\lambda_\alpha \in V_0(ae_\alpha)} |\lambda_\alpha - f(a)|,
\]
from which it follows that
\[
d_H(V_0(ae_\alpha),\{f(a)\})
= \sup_{\lambda_\alpha \in V_0(ae_\alpha)} |\lambda_\alpha - f(a)|.
\]

Thus, proving
\(
\lim_{\alpha}d_H\bigl(V_0(ae_\alpha),\{f(a)\}\bigr)= 0
\)
reduces to checking
\(
\lim_{\alpha}\sup_{\lambda_\alpha \in V_0(ae_\alpha)} |\lambda_\alpha - f(a)| = 0.
\)
Since $f$ is a pure state, we have $f(a^*a) \ge 0$. We proceed by considering two separate cases.

\textbf{Case 1.} $f(a^*a) = 0$.

Suppose that $f(a^*a) = 0$. By Cauchy-Schwarz inequality we obtain $0 \le |f(a)|^2 \le f(1) f(a^*a) = 0$. Hence, $f(a) = 0$.
Also,
\[
\begin{aligned}
 0 \le \|a e_\alpha\|^2 &= \|(a e_\alpha)^*(a e_\alpha)\| = \|e_\alpha a^* a e_\alpha\| \\
&= \left\| e_\alpha a^* a e_\alpha - f(a^*a) e_\alpha^2 + f(a^*a) e_\alpha^2 \right\| \\
&\le \left\| e_\alpha a^* a e_\alpha - f(a^*a) e_\alpha^2 \right\| + \left\| f(a^*a) e_\alpha^2 \right\|.
\end{aligned}
\]
Consequently, $\lim_{\alpha}\|a e_\alpha\| = 0$.
For each \(\alpha\), fix $\alpha$ and $\lambda_\alpha \in V_0(ae_\alpha)$, there exists $\rho_{\alpha,\lambda} \in \mathcal{S}_{\max}(ae_\alpha)$ such that $\lambda_\alpha = \rho_{\alpha,\lambda}(ae_\alpha)$
and
\(
\rho_{\alpha,\lambda}(e_\alpha a^* a e_\alpha) = \|a e_\alpha\|^2.
\)
Applying the Cauchy-Schwarz inequality, we obtain
\[
0\le |\lambda_\alpha|^2 =|\rho_{\alpha,\lambda}(ae_\alpha)|^2 \le \rho_{\alpha,\lambda}(1) \, \rho_{\alpha,\lambda}(e_\alpha a^* a e_\alpha) = \|a e_\alpha\|^2.
\]
This entails that $0 \leq |\lambda_\alpha| \leq \|a e_\alpha\|$ and thus $\sup_{\lambda_\alpha \in V_0(ae_\alpha)}|\lambda_\alpha-f(a)| \leq \|a e_\alpha\|$.
Consequently,
$$0 \le d_H(V_0(a e_\alpha), \{f(a)\})=\sup_{\lambda_\alpha \in V_0(ae_\alpha)}|\lambda_\alpha-f(a)| \leq \|a e_\alpha\|.$$
Taking limits in the above inequality yields
\[
\lim_{\alpha}d_H(V_0(a e_\alpha), \{f(a)\}) = 0.
\]

\textbf{Case 2.} $f(a^*a) > 0$.

Now suppose  $f(a^*a) > 0$.
We first claim that $\lim_{\alpha}\|a e_\alpha\|^2 = f(a^*a)$.
By the assumption in the proposition,
\[
\lim_{\alpha} \|e_\alpha a^*a e_\alpha - f(a^*a) e_\alpha^2\| = 0.
\]
Since $e_\alpha$ is positive and $\|e_\alpha\| = 1$, we have $\|e_\alpha^2\| = \|e_\alpha^* e_\alpha\| = \|e_\alpha\|^2 = 1$. Then
\[
\big| \|e_\alpha a^* a e_\alpha\| - f(a^*a) \big|
= \left| \|e_\alpha a^* a e_\alpha\| - f(a^*a) \|e_\alpha^2\| \right|
\le \| e_\alpha a^* a e_\alpha - f(a^*a) e_\alpha^2 \|.
\]
Hence $\lim_{\alpha}\|a e_\alpha\|^2 = f(a^*a)$.

For each \(\alpha\), pick any $\lambda_\alpha \in V_0(ae_\alpha)$. Then there exists $\rho_{\alpha,\lambda} \in \mathcal{S}_{\max}(a e_\alpha)$ such that $\lambda_\alpha = \rho_{\alpha,\lambda}(a e_\alpha)$ and
\(
\rho_{\alpha,\lambda}(e_\alpha a^* a e_\alpha) = \|a e_\alpha\|^2.
\)
We claim that $\lim_{\alpha}\rho_{\alpha,\lambda}(e_\alpha^2) = 1$. Indeed, since $\rho_{\alpha,\lambda}$ is a state,
\[
\begin{aligned}
& \left| \rho_{\alpha,\lambda}(e_\alpha a^* a e_\alpha) - f(a^*a) \rho_{\alpha,\lambda}(e_\alpha^2) \right|\\
=& \left| \rho_{\alpha,\lambda}\left(e_\alpha a^* a e_\alpha - f(a^*a) e_\alpha^2\right) \right|
\le \left\| e_\alpha a^* a e_\alpha - f(a^*a) e_\alpha^2 \right\|.
\end{aligned}
\]
Thus,
\(
\lim_{\alpha} \big| \|a e_\alpha\|^2  -  f(a^*a) \rho_{\alpha,\lambda}(e_\alpha^2) \big|  = 0.
\)
Since $\lim_{\alpha}\|a e_\alpha\|^2 = f(a^*a)$, we have
\[
f(a^*a)\left|\rho_{\alpha,\lambda}(e_\alpha^2)-1\right|\leq
\left|f(a^*a)\rho_{\alpha,\lambda}(e_\alpha^2)-\|ae_\alpha\|^2\right|+\left|\|ae_\alpha\|^2-f(a^*a)\right|.
\]
It follows that
\(
\lim_{\alpha}f(a^*a)\left|\rho_{\alpha,\lambda}(e_\alpha^2)-1\right| = 0.
\)
As $f(a^*a) > 0$, we get $\lim_{\alpha}\rho_{\alpha,\lambda}(e_\alpha^2) = 1$, as desired.

Next we show that $\lim_{\alpha}|\rho_{\alpha,\lambda}(a e_\alpha) - \rho_{\alpha,\lambda}(e_\alpha a e_\alpha)| = 0$.
Note that
\(
a e_\alpha = e_\alpha a e_\alpha + (a e_\alpha - e_\alpha a e_\alpha) = e_\alpha a e_\alpha + (1_\mathcal{A} - e_\alpha) a e_\alpha.
\)
Then we have
\(
\rho_{\alpha,\lambda}(a e_\alpha) - \rho_{\alpha,\lambda}(e_\alpha a e_\alpha) = \rho_{\alpha,\lambda}\left( (1_\mathcal{A} - e_\alpha) a e_\alpha \right).
\)
By  Cauchy-Schwarz inequality again, one gets
\[
\left| \rho_{\alpha,\lambda}\left( (1_\mathcal{A} - e_\alpha) a e_\alpha \right) \right|^2
\le \rho_{\alpha,\lambda}\left( (1_\mathcal{A} - e_\alpha)^2 \right) \rho_{\alpha,\lambda}\left( e_\alpha a^* a e_\alpha \right).
\]
Since $\rho_{\alpha,\lambda}(e_\alpha a^* a e_\alpha) = \|a e_\alpha\|^2$ is bounded, to prove $|\rho_{\alpha,\lambda} ((1_\mathcal{A} - e_\alpha) a e_\alpha)|$ converges to 0, it suffices to show that $\lim_{\alpha}\rho_{\alpha,\lambda}((1_\mathcal{A} - e_\alpha)^2) = 0$.
Since $0 \le e_\alpha \le 1_\mathcal{A}$, we have $ e_\alpha^2 \le e_\alpha$ and
\(
(1_\mathcal{A} - e_\alpha)^2 \le  1_\mathcal{A} - e_\alpha.
\)
Therefore,
\[
0 \le \rho_{\alpha,\lambda}\left( (1_\mathcal{A} - e_\alpha)^2 \right) \le \rho_{\alpha,\lambda}(1_\mathcal{A} - e_\alpha) = 1 - \rho_{\alpha,\lambda}(e_\alpha) \le 1 - \rho_{\alpha,\lambda}(e_\alpha^2).
\]
Consequently,
\[
\left|
\rho_{\alpha,\lambda}
\bigl((1_{\mathcal A}-e_\alpha)ae_\alpha\bigr)
\right|
\leq
\|ae_\alpha\|
\left(
1-\rho_{\alpha,\lambda}(e_\alpha^2)
\right)^{1/2},
\]
which leads to $\lim_{\alpha} |\rho_{\alpha,\lambda}(a e_\alpha) - \rho_{\alpha,\lambda}(e_\alpha a e_\alpha)| = 0$.

Finally, we claim that $\lim_{\alpha}\rho_{\alpha,\lambda}(e_\alpha a e_\alpha) = f(a)$.
In fact,
\begin{align*}
0\leq \left| \rho_{\alpha,\lambda}(e_\alpha a e_\alpha) - f(a) \right|
&= \left| \rho_{\alpha,\lambda}(e_\alpha a e_\alpha) - f(a) \rho_{\alpha,\lambda}(e_\alpha^2) + f(a) \rho_{\alpha,\lambda}(e_\alpha^2) - f(a) \right| \\
&\le \left| \rho_{\alpha,\lambda}(e_\alpha a e_\alpha) - f(a) \rho_{\alpha,\lambda}(e_\alpha^2) \right| + \left| f(a) \rho_{\alpha,\lambda}(e_\alpha^2) - f(a) \right| \\
&= \left| \rho_{\alpha,\lambda}\big( e_\alpha a e_\alpha - f(a) e_\alpha^2 \big) \right| + \left| f(a) \right| \left| \rho_{\alpha,\lambda}(e_\alpha^2) - 1 \right|\\
&\le \left\| e_\alpha a e_\alpha - f(a) e_\alpha^2 \right\|+ \left| f(a) \right| \left| \rho_{\alpha,\lambda}(e_\alpha^2) - 1 \right|,
\end{align*}
which ensures that
$\lim_{\alpha}\rho_{\alpha,\lambda}(e_\alpha a e_\alpha) = f(a)$.
Then
\begin{align*}
0\leq \left| \rho_{\alpha,\lambda}(a e_\alpha) - f(a) \right|
&= \left| \rho_{\alpha,\lambda}(a e_\alpha) - \rho_{\alpha,\lambda}(e_\alpha a e_\alpha) + \rho_{\alpha,\lambda}(e_\alpha a e_\alpha) - f(a) \right| \\
&\le \left| \rho_{\alpha,\lambda}(a e_\alpha) - \rho_{\alpha,\lambda}(e_\alpha a e_\alpha) \right| + \left| \rho_{\alpha,\lambda}(e_\alpha a e_\alpha) - f(a) \right| \\
&= \left| \rho_{\alpha,\lambda}\left( (1_\mathcal{A} - e_\alpha)a e_\alpha \right) \right|+\left| \rho_{\alpha,\lambda}(e_\alpha a e_\alpha) - f(a) \right|.
\end{align*}
It follows that
$$\sup_{\lambda_\alpha \in V_0(ae_\alpha)} |\lambda_\alpha - f(a)|\le\sup_{\rho_{\alpha,\lambda}\in\mathcal S_{\max} (ae_\alpha)}(\left| \rho_{\alpha,\lambda}\left( (1_\mathcal{A} - e_\alpha)a e_\alpha \right) \right|+\left| \rho_{\alpha,\lambda}(e_\alpha a e_\alpha) - f(a) \right|).$$
Then,
\[
0 \le d_H(V_0(a e_\alpha), \{f(a)\})\le \sup_{\rho_{\alpha,\lambda}\in\mathcal S_{\max} (ae_\alpha)}
(\left| \rho_{\alpha,\lambda}\left( (1_\mathcal{A} - e_\alpha)a e_\alpha \right) \right|+\left| \rho_{\alpha,\lambda}(e_\alpha a e_\alpha) - f(a) \right|).
\]

Put
\[
\eta_\alpha
=
\left\|
e_\alpha a^*ae_\alpha
-
f(a^*a) e_\alpha^2
\right\|,
r_\alpha
=
\left|
\|ae_\alpha\|^2-f(a^*a)
\right|
\]
and
\[
\delta_\alpha
=
\left\|
e_\alpha ae_\alpha
-
f(a)e_\alpha^2
\right\|.
\]
By the assumption on $(e_\alpha)_{\alpha\in\Lambda}$, we have
\[
\eta_\alpha\to 0,
r_\alpha\to 0,
\delta_\alpha\to 0.
\]
For each \(\lambda_\alpha \in V_0(ae_\alpha)\) and corresponding $\rho_{\alpha,\lambda}$,
\[
\begin{aligned}
f(a^*a)\left|
\rho_{\alpha,\lambda}(e_\alpha^2)-1
\right|
&\leq
\left|
f(a^*a)\rho_{\alpha,\lambda}(e_\alpha^2)
-
\|ae_\alpha\|^2
\right|
+
\left|
\|ae_\alpha\|^2-f(a^*a)
\right|  \\
&=
\left|
\rho_{\alpha,\lambda}
\left(
f(a^*a)e_\alpha^2-e_\alpha a^*ae_\alpha
\right)
\right|
+
r_\alpha \\
&\leq
\eta_\alpha+r_\alpha.
\end{aligned}
\]
Hence we get
\[
\left|
\rho_{\alpha,\lambda}(e_\alpha^2)-1
\right|
\leq
\frac{\eta_\alpha+r_\alpha}{f(a^*a)}
\]
and
\[
\begin{aligned}
\left|
\rho_{\alpha,\lambda}((1_{\mathcal A}-e_\alpha)ae_\alpha)
\right|
&\leq
\|ae_\alpha\|
\left(
1-\rho_{\alpha,\lambda}(e_\alpha^2)
\right)^{1/2}
\leq
\|ae_\alpha\|
\left(
\frac{\eta_\alpha+r_\alpha}{f(a^*a)}
\right)^{1/2}.
\end{aligned}
\]

Also note that
\[
\begin{aligned}
\left|
\rho_{\alpha,\lambda}(e_\alpha ae_\alpha)-f(a)
\right|
&\leq
\left|
\rho_{\alpha,\lambda}
\left(
e_\alpha ae_\alpha-f(a)e_\alpha^2
\right)
\right|
+
|f(a)|
\left|
\rho_{\alpha,\lambda}(e_\alpha^2)-1
\right|  \\
&\leq
\delta_\alpha
+
|f(a)|
\frac{\eta_\alpha+r_\alpha}{f(a^*a)}.
\end{aligned}
\]
It follows that
$$\begin{array}{rl}
0 \le d_H(V_0(a e_\alpha), \{f(a)\})\le & \sup_{\rho_{\alpha,\lambda}\in\mathcal S_{\max} (ae_\alpha)}
(\left| \rho_{\alpha,\lambda}\left( (1_\mathcal{A} - e_\alpha)a e_\alpha \right) \right|+\left| \rho_{\alpha,\lambda}(e_\alpha a e_\alpha) - f(a) \right|)\\
\le & \|ae_\alpha\|
\left(
\frac{\eta_\alpha+r_\alpha}{f(a^*a)}
\right)^{1/2}+\delta_\alpha
+
|f(a)|
\frac{\eta_\alpha+r_\alpha}{f(a^*a)}.
\end{array}$$
Consequently, we achieve
\[
\lim_{\alpha}d_H(V_0(a e_\alpha), \{f(a)\}) = 0.
\]

The proof is complete.
\end{proof}

The following result indicates  that  every element in a $C^*$-algebra can be uniquely determined by the algebraic maximal numerical range of its products with other elements.

\begin{Proposition} \label{prop2.3} Let \(\mathcal A\) be a unital \(C^*\)-algebra, and let
\(a,b\in\mathcal A\). If
$
V_0(ax)=V_0(bx)$ holds for all $ x\in\mathcal A$,
 then \(a=b\).
\end{Proposition}

\begin{proof} Fix a pure state \(f\) on $\mathcal A$. By Proposition \ref{prop2.1}, there exists a decreasing net \((e_\alpha)_{\alpha\in\Lambda}\subseteq  \mathcal A_+\) such that
\(
\lim_{\alpha} \|e_\alpha x e_\alpha - f(x) e_\alpha^2\| = 0
\text{~for every~} x\in\mathcal A.
\)
Applying Proposition \ref{prop2.2} to \(a, b\in \mathcal A\) gives
\(
\lim_{\alpha}d_H\bigl(V_0(ae_\alpha),\{f(a)\}\bigr)= 0
\)
and
\(
\lim_{\alpha}d_H\bigl(V_0(be_\alpha),\{f(b)\}\bigr)= 0.
\)
Since \(V_0(ax)=V_0(bx)\) for every \(x\in\mathcal A\),
\[
V_0(ae_\alpha)=V_0(be_\alpha),~\alpha \in \Lambda .
\]
By the triangle inequality for the Hausdorff distance, one gets
\[
d_H(\{f(a)\}, \{f(b)\}) \leq d_H(\{f(a)\}, V_0(ae_\alpha)) + d_H(V_0(ae_\alpha), V_0(be_\alpha)) + d_H(V_0(be_\alpha), \{f(b)\}).
\]
The middle term vanishes because $V_0(ae_\alpha) = V_0(be_\alpha)$. Thus,
\[
\begin{aligned}
 d_H(\{f(a)\}, \{f(b)\}) &\le d_H(\{f(a)\}, V_0(ae_\alpha)) + d_H(V_0(be_\alpha), \{f(b)\})\\
 &\le  d_H(V_0(ae_\alpha), \{f(a)\})+d_H(V_0(be_\alpha), \{f(b)\}).
\end{aligned}
\]
Since the distance is nonnegative and the right-hand side tends to $0$,
 we obtain
\[
|f(a) - f(b)|=d_H(\{f(a)\}, \{f(b)\})= 0,
\]
which implies
\(
f(a-b)=0.
\)
Recall that the state space \(\mathcal{S}(\mathcal A)\) is the weak\(^*\)-closed convex hull of its extreme points, and the extreme points of $\mathcal{S}(\mathcal A)$ are precisely the pure states. Since the map $\varphi\mapsto \varphi(a-b)$ is affine and weak*-continuous on $\mathcal{S}(\mathcal A)$,
it follows from the equality on pure states that $\varphi(a-b)=0$ for all $\varphi\in \mathcal{S}(\mathcal A)$. Hence we must have \(a=b\).
\end{proof}

The following lemma is well known.

\begin{Proposition}\label{prop2.4}
Let \(T\in \mathcal{B}(H)\) with \(0 \le T \le \lambda I\) for some \(\lambda \ge 0\) and \(x\in H\) . If \(\|x\| = 1\) and \(\langle Tx, x \rangle = \lambda\), then \(Tx= \lambda x\).
\end{Proposition}

\begin{proof}
Since \(0 \le T \le \lambda I\), the operator \(P = \lambda I - T\ge 0\). Hence,
\(
\langle Px, x \rangle \ge 0
\)
for every vector \(x\in H\).
On the other hand,
\(
\langle Px, x \rangle
= \langle (\lambda I - T)x, x \rangle
= \lambda \langle x, x \rangle - \langle Tx, x \rangle=0.
\)
Since \(P \ge 0\), it admits a unique positive square root \(P^{1/2}\) such that
\[
\langle Px, x \rangle
= \langle P^{1/2} P^{1/2} x, x \rangle
= \langle P^{1/2} x, P^{1/2} x \rangle
= \left\| P^{1/2} x \right\|^2.
\]
It implies \(P^{1/2} x = 0\). Multiplying both sides by \(P^{1/2}\) gives
\(
Px = P^{1/2} \left( P^{1/2} x \right) = 0.
\)
Therefore \(Tx = \lambda x\). This completes the proof.
\end{proof}

We do not know how to use the algebraic maximal numerical range to characterize all self-adjoint elements in $C^*$-algebras. But we have the following result, which is very useful in the next section.

\begin{Proposition}\label{prop2.5}
Let \(\mathcal{A}\) be a unital \(C^*\)-algebra and \(a\in \mathcal{A}\) satisfies \( a^2=1_{\mathcal A}. \) If \( V_0(a)\neq \{0\} \), then \( a=a^*.\)
\end{Proposition}

\begin{proof}
Since \(a^2=1_{\mathcal A}\), we have \(a^{-1}=a.\)
Hence \( 1=\|a^{2}\|\leq \|a\|^2 \). Thus \( \|a\|\geq 1. \) In what follows, we discuss two distinct cases.

We first consider the case \(\|a\|=1\). Put \(h=a^*a\). Then \(h\geq0\) and
\(\|h\|=\|a^*a\|=\|a\|^2=1\). Moreover,
\(
h^{-1}=(a^*a)^{-1}=a^{-1}(a^*)^{-1}=aa^*,
\)
and therefore \(\|h^{-1}\|=\|aa^*\|=\|a\|^2=1\). Since \(h\geq0\) and
\(\|h\|=1\), we have \(\sigma(h)\subseteq[0,1]\). Similarly,
\(h^{-1}\geq0\) and \(\|h^{-1}\|=1\) imply
\(\sigma(h^{-1})\subseteq[0,1]\).
Since
\(
\sigma(h^{-1})=\{\lambda^{-1}:\lambda\in\sigma(h)\},
\)
it follows that \(\sigma(h)\subseteq[1,\infty)\). Hence
\(
\sigma(h)=\{1\}.
\)
As \(h\) is positive, \(h=1_{\mathcal A}\).
Thus \(a^*a=1_{\mathcal A}\), and consequently \(a=a^*\).

It remains to exclude the case \(\|a\|>1\). For any $f\in \mathcal S_{\max} (a)$, let
\((\pi_f, H_f, e)\) be the GNS representation of $\mathcal A$ associated with \(f\), where
\(\|e\|=1\) and
\[
f(x)=\langle \pi_f(x)e,e\rangle\text{~for all~} x \in \mathcal{A}.
\]
The representation is a $*$-homomorphism that satisfies \(\pi_f(1_\mathcal{A}) = I_{H_f}\).
Since \(\pi_f\) is contractive, we have
\[
0 \le \pi_f(a)^* \pi_f(a) = \pi_f(a^*a) \le \left\| \pi_f(a^*a) \right\| I_{H_f} \le  \|a\|^2 I_{H_f}.
\]
Moreover,
\[
\langle \pi_f(a)^* \pi_f(a) \, e, e \rangle
= \langle \pi_f(a^*a) e, e \rangle
= f(a^*a) = \|a\|^2.
\]
It follows from Proposition \ref{prop2.4} that
\[
\pi_f(a)^* \pi_f(a) e = \|a\|^2  e.
\]
As \(\pi_f(a^2)=\pi_f(1_\mathcal{A}) = I_{H_f}\), we also have \(\pi_f(a)^2= I_{H_f}\) and \(\pi_f(a)^2e=e\). Hence
\[
\begin{aligned}
f(a) = \langle \pi_f(a) e, e \rangle
&= \langle \pi_f(a) e, \pi_f(a)^2 e \rangle\\
&= \langle \pi_f(a)^* \pi_f(a) e, \pi_f(a) e \rangle \\
&= \|a\|^2 \langle e, \pi_f(a) e \rangle
= \|a\|^2 \overline{\langle \pi_f(a) e, e \rangle}
= \|a\|^2 \overline{f(a)}.
\end{aligned}
\]
Since \(\|a\| > 1\), we have \(1 - \|a\|^2 \neq 0\), which forces \(f(a) = 0\).
Thus every \(f\in \mathcal{S}_{\max}(a)\) satisfies \(f(a)=0\), and hence
\(V_0(a)=\{0\}\). This contradicts the hypothesis \(V_0(a)\neq \{0\}\).
Therefore the case \(\|a\|>1\) is impossible.

So we must have \(\|a\| = 1\) and \(a = a^*\).
\end{proof}

\begin{Proposition}\label{prop2.6}
Let $\mathcal{A}$ be a unital $C^*$-algebra, and let
$a\in\mathcal{A}$ be hyponormal. Then
\(
V_0(a) = co\bigl(\sigma_n(a)\bigr),
\)
where $\sigma_n(a) := \{\lambda \in \sigma(a) : |\lambda| = \|a\|\}.$ In particular, for positive elements $a\in\mathcal{A}$, one has \(V_0(a)=\{\|a\|\}\).
\end{Proposition}

\begin{proof}
For hyponormal elements, the conclusion cames from  \cite[Proposition 2.3]{benabdi2021maximal}. Assume that $a\in\mathcal A$ is a positive element. Then
\(
\sigma(a)\subseteq[0,\|a\|]
\) and
\(
\sigma_n(a)
=
\left\{
\lambda\in\sigma(a):
|\lambda|=\|a\|
\right\}
=
\{\|a\|\}.
\)
Consequently, by applying the conclusion for hyponormal,
we get \(V_0(a)=\{\|a\|\}\), as desired.
\end{proof}

\begin{Proposition}\label{prop2.7}
Let $\mathcal{A}$ be a unital \(C^*\)-algebra. If \(\{a_n\}_{n=1}^\infty\) is a sequence in $\mathcal{A}$ such that \(\|a_n-a\|\to 0\)  for some $a\in\mathcal A$ and  \(\{\lambda_n\}_{n=1}^\infty \) is a sequence in \(\mathbb{C}\) such that \(\lambda_n\in V_0(a_n)\)  for every \(n\) and \(\lambda_n\to\lambda\) for some $\lambda\in\mathbb{C}$, then \(\lambda\in V_0(a)\).
\end{Proposition}

\begin{proof}
For each $n\geq1$, since \(\lambda_n\in V_0(a_n)\), there exists \(f_n\in \mathcal{S}(\mathcal A)\) such that
\[
f_n(a_n) = \lambda_n ~\text{and}~ f_n(a_n^* a_n) = \|a_n\|^2.
\]
Note that the state space \(\mathcal{S}(\mathcal A)\) is weak $^*$-compact.
Since  $\{f_n\}$ is a sequence in $\mathcal{S}(\mathcal{A})$, it admits a weak $^*$-convergent subsequence $\{f_{n_k}\}$. Hence,
\(
f_{n_k}
\xrightarrow{wk^*}
f
\)
for some $f\in\mathcal S(\mathcal A)$. It follows that
\(
\|a_{n_k}-a\|\to 0,  \lambda_{n_k} \to \lambda,
\)
and
\[
f_{n_k}(a_{n_k}) = \lambda_{n_k}, ~ f_{n_k}(a_{n_k}^* a_{n_k}) = \|a_{n_k}\|^2.
\]
\if false We now prove that $f(a) = \lambda$, $f \in \mathcal{S}_{\mathrm{max}}(a)$ and thus $\lambda \in V_0(a)$.\fi

As
\[
\begin{aligned}
\bigl| f_{n_k}(a_{n_k}) - f(a) \bigr|
&\le \bigl| f_{n_k}(a_{n_k}) - f_{n_k}(a) \bigr| + \bigl| f_{n_k}(a) - f(a) \bigr|\\
&\le \|a_{n_k} - a\|+\bigl| f_{n_k}(a) - f(a) \bigr|.
\end{aligned},
\]
 $\|a_{n_k} - a\| \to 0$  and
\(
f_{n_k}
\xrightarrow{wk^*}
f
\), we have
\[
f_{n_k}(a_{n_k}) \to f(a).
\]
On the other hand, $f_{n_k}(a_{n_k}) = \lambda_{n_k} \to \lambda$. Hence,
\(
f(a) = \lambda.
\)

Next we check that $f \in \mathcal{S}_{\mathrm{max}}(a)$.

Clearly, if \(\|a_n-a\|\to 0\) and \(\|b_n-b\|\to 0\), then
 \(\|a_n^*-a^*\|\to 0\),
 $\|a_nb_n - ab\|\to 0$ and
 $\|a_n\| \to \|a\|$.
Then
\[
\begin{aligned}
\bigl| f_{n_k}(a_{n_k}^* a_{n_k}) - f(a^* a) \bigr|
&\le \bigl| f_{n_k}(a_{n_k}^* a_{n_k}) - f_{n_k}(a^* a) \bigr| + \bigl| f_{n_k}(a^* a) - f(a^* a) \bigr|\\
&\le\| a_{n_k}^* a_{n_k} - a^* a \|+ \bigl| f_{n_k}(a^* a) - f(a^* a) \bigr|.
\end{aligned}
\]
Since $\| a_{n_k}^* a_{n_k} - a^* a\| \to0$ and
\(
f_{n_k}
\xrightarrow{wk^*}
f
\), it follows that
\[
\|a_{n_k}\|^2=f_{n_k}(a_{n_k}^* a_{n_k}) \to f(a^* a).
\]
But $\|a_{n_k}\|^2 \to \|a\|^2$, which entails
\(
f(a^* a) = \|a\|^2.
\)
Therefore, one has $f \in \mathcal{S}_{\mathrm{max}}(a)$.

Now $\lambda = f(a)$ and $f \in \mathcal{S}_{\mathrm{max}}(a)$ imply $\lambda \in V_0(a)$. This completes the proof.
\end{proof}

\section {Preservers of ordered Triple products}

In this section, we discuss the general preserver problem of algebraic maximal numerical range on $C^*$-algebras.
We first  investigate surjective maps between unital $C^*$-algebras that preserve the algebraic maximal numerical range of triple products of the form \(abc\).

The following result is basic.

\begin{Theorem}\label{thm3.1}
Let \(\mathcal A\) and \(\mathcal B\) be unital \(C^*\)-algebras,
and let \(\Phi:\mathcal {A} \to \mathcal {B}\) be a surjective map such that
\[
V_0(\Phi(a)\Phi(b)\Phi(c))=V_0(abc) \text{~for all~} a, b, c \in \mathcal{A}.
\]
Then \(\Phi\) is injective, \(\Phi(1_\mathcal{A})\in Z(\mathcal B)\) is   invertible and  the map
\(
\Psi:\mathcal A\to\mathcal B
\)
defined by
\[
\Psi(a)=\Phi(1_{\mathcal A})^{-1}\Phi(a)\text{~for all~} a\in \mathcal{A}
\]
is  a  multiplicative bijection.
\end{Theorem}

\begin{proof}
We prove the result by  several claims. Put \(u=\Phi(1_{\mathcal A})\).

 \textbf{Claim 1.} $\Phi$ is injective.

Suppose that \(a,b\in\mathcal A\)
and \(\Phi(a)=\Phi(b)\).
For every \(x\in\mathcal A\), the hypothesis applied to the triples \((a,x,1_{\mathcal A})\) and \((b,x,1_{\mathcal A})\) gives
\[
V_0(\Phi(a)\Phi(x)u)=V_0(ax)
\]
and
\[
V_0(\Phi(b)\Phi(x)u)=V_0(bx).
\]
Since \(\Phi(a)=\Phi(b)\), it follows that
\[
V_0(ax)=V_0(bx)\text{~for all~} x\in \mathcal{A}.
\]
By Proposition \ref{prop2.3}, one gets \(a=b\). Hence \(\Phi\) is injective.

\textbf{Claim 2.} \(u\in Z(\mathcal{B})\).

Let \(a,b\in\mathcal A\). Applying the hypothesis to \((1_{\mathcal A},a,b)\) and
\((a,1_{\mathcal A},b)\), one obtains
\[
V_0(u\Phi(a)\Phi(b))=V_0(ab)
\]
and
\[
V_0(\Phi(a)u\Phi(b))=V_0(ab).
\]
Thus
\[
V_0(u\Phi(a)\Phi(b))=V_0(\Phi(a)u\Phi(b))\text{~for all~} b \in \mathcal{A}.
\]
Since \(\Phi\) is surjective, \(\Phi(b)\) runs through all elements of
\(\mathcal B\).
Applying Proposition \ref{prop2.3} again, we see that
\(
u\Phi(a)=\Phi(a)u
\) holds for all $a\in\mathcal A$. Now the surjectivity of $\Phi$ forces
\(
u\in Z(\mathcal B).
\)

 \textbf{Claim 3.} $\Phi(a)\Phi(b)=\Phi(ab)u = u\Phi(ab)\text{~for all~} a, b \in \mathcal{A}$.

Let \(a,b,c\in\mathcal A\). The hypothesis gives
\[
V_0(\Phi(a)\Phi(b)\Phi(c))=V_0(abc).
\]
Applying the same hypothesis to the triple \((ab,1_{\mathcal A},c)\), we also have
\[
V_0(\Phi(ab)u\Phi(c))=V_0(abc).
\]
Therefore
\[
V_0(\Phi(a)\Phi(b)\Phi(c))
=
V_0(\Phi(ab)u\Phi(c))\text{~for all~} c \in \mathcal{A}.
\]
Since \(\Phi(c)\) runs over all elements of
\(\mathcal B\), Proposition \ref{prop2.3} entails
\(
\Phi(a)\Phi(b)=\Phi(ab)u.
\)
Since \(u\in Z(\mathcal B)\), we also have
\(
\Phi(ab)u=u\Phi(ab).
\)

 \textbf{Claim 4.} $u$ is invertible.

Since \(\Phi\) is surjective,
there exists \(e\in\mathcal A\) such that
\(
\Phi(e)=1_{\mathcal B}.
\)
Taking \(b=e\) in Claim 3, we get
\[
\Phi(ae)u=\Phi(a)\Phi(e)=\Phi(a) \text{~for all~} a \in \mathcal{A}.
\]
For \(y\in\mathcal B\), there exists
\(a_1\in\mathcal A\) such that \(y=\Phi(a_1)\). Hence
\(
y=\Phi(a_1e)u\in \mathcal B u.
\)
Therefore \(\mathcal B \subseteq \mathcal B u\subseteq \mathcal B\). Therefore \(\mathcal B = \mathcal B u\). In particular,
\(1_{\mathcal B}\in \mathcal B u\), so there exists \(v\in\mathcal B\) such that
\(
vu=1_{\mathcal B}.
\)
Since \(u\in Z(\mathcal B)\), we also have
\(
uv=vu=1_{\mathcal B}.
\)
Thus \(u\) is invertible and \(u^{-1}=v\).

\textbf{Claim 5.}  $\Psi = u^{-1}\Phi$ is unital, multiplicative and bijective.

Clearly,
\(
\Psi(1_{\mathcal A})=u^{-1}\Phi(1_{\mathcal A})=1_{\mathcal B}.
\)
Since \(u^{-1}\in Z(\mathcal B)\), by Claim 3, for any \(a,b\in\mathcal A\), we have
\[
\begin{aligned}
\Psi(a)\Psi(b)
&=u^{-1}\Phi(a)u^{-1}\Phi(b)  \\
&=u^{-2}\Phi(a)\Phi(b)  =u^{-2}\Phi(ab)u  \\
&=u^{-1}\Phi(ab)  =\Psi(ab).
\end{aligned}
\]
So $\Psi(ab)=\Psi(a)\Psi(b)$ holds for all $a,b\in \mathcal A$.

The bijectivity of $\Psi$ is obvious.

\if false  \textbf{Claim 6.} $\Psi$ is bijective.

If
$\Psi(a)=\Psi(b)$, then
\(
    \Phi(a)=u\Psi(a)
             =u\Psi(b)
             =\Phi(b),
\)
and the injectivity of $\Phi$ implies $a=b$. Hence \(\Psi\) is injective. To prove
surjectivity, let $b\in\mathcal B$, since $\Phi$ is surjective, there exists
$a\in\mathcal A$ such that $\Phi(a)=u b$. Hence
\(
    u\Psi(a)=\Phi(a)=u b,
\)
and thus $\Psi(a)=b$. Consequently, $\Psi$ is bijective.\fi

The proof is complete.
\end{proof}

The next result is a generalization of Theorem \ref{thm3.1}.

\begin{Theorem}\label{thm3.2}
Let $\mathcal A$ and $\mathcal B$ be unital $C^*$-algebras, and let
$\Phi:\mathcal A\to\mathcal B$ be a surjective map. Suppose that for some fixed integer $n\in \mathbb{N}$ with $n\geq 3$,
\[
V_0\bigl(\Phi(a_1)\Phi(a_2)\cdots \Phi(a_n)\bigr)
=
V_0(a_1a_2\cdots a_n) \text{~for all~} a_1,\ldots,a_n\in\mathcal {A}.
\]
Put \(u=\Phi(1_{\mathcal A})\). Then the following assertions hold.
\begin{enumerate}
    \item $\Phi$ is injective.
    \item \(u\) is an invertible element in $Z(\mathcal B)$.
    \item For every integer $m\geq 1$ and every
    $a_1,\ldots,a_m\in\mathcal A$,
    \[
    \Phi(a_1)\Phi(a_2)\cdots\Phi(a_m)
    =
    \Phi(a_1a_2\cdots a_m)u^{\,m-1}=u^{m-1}\Phi(a_1a_2\cdots a_m).
    \]
    \item The map
    \(
    \Psi:\mathcal A\to\mathcal B\) defined by
    \(\Psi=u^{-1}\Phi
    \)
    is a unital multiplicative bijection.
\end{enumerate}
\end{Theorem}

\begin{proof}
We first prove that \(\Phi\) is injective. Suppose that \(a,b\in \mathcal A\) and
\(\Phi(a)=\Phi(b)\). For every \(x\in \mathcal A\), applying the hypothesis to
\(
(a,x,1_{\mathcal A},\ldots,1_{\mathcal A})
        ~\text{and}~
(b,x,1_{\mathcal A},\ldots,1_{\mathcal A}),
\)
where \(1_{\mathcal A}\) appears \(n-2\) times,  one gets
\(
V_0(ax)=V_0(bx)\text{~for all x $\in \mathcal{A}$}.
\)
By Proposition \ref{prop2.3}, we have \(a=b\). Thus \(\Phi\) is injective.

Since $\Phi$ is surjective, there exists $e\in\mathcal A$ such that
\(
\Phi(e)=1_{\mathcal B}.
\)
Next, we show that $u\in Z(\mathcal B)$. Applying the hypothesis to
\(
(1_{\mathcal A},a,c,e,\ldots,e)
~\text{and}~
(a,1_{\mathcal A},c,e,\ldots,e)
\)
gives
\(
V_0\bigl(u\Phi(a)\Phi(c)\bigr)
=
V_0\bigl(\Phi(a)u\Phi(c)\bigr).
\)
As $c$ runs through $\mathcal A$, the element $\Phi(c)$ runs through all of
$\mathcal B$. By Proposition \ref{prop2.3}, it follows that
\(
u\Phi(a)=\Phi(a)u.
\)
Since $a\in\mathcal A$ was arbitrary and $\Phi$ is surjective, it follows
that
\(
u\in Z(\mathcal B).
\)

 Let $a,b,c\in\mathcal A$ be arbitrary.
Applying hypothesis to
\(
(a,b,c,e,\ldots,e)
~\text{and}~
(ab,1_{\mathcal A},c,e,\ldots,e),
\)
we get
\(
V_0\bigl(\Phi(a)\Phi(b)\Phi(c)\bigr)
=
V_0\bigl(\Phi(ab)u\Phi(c)\bigr).
\)
By Proposition \ref{prop2.3} and note that $u\in Z(\mathcal B)$, we get
$$\label{eq3.1}
\Phi(a)\Phi(b)=\Phi(ab)u=u\Phi(ab). \eqno(3.1)
$$

We now show that \(u\) is invertible. Taking \(b=e\) in the identity just
proved gives \(\Phi(a)=\Phi(ae)u\) for all \(a\in \mathcal A\).
Similar to the proof of Theorem \ref{thm3.1}, we see that
\(
\mathcal B=\mathcal B u.
\)
In particular, $1_{\mathcal B}\in\mathcal B u$, and hence there exists
$v\in\mathcal B$ such that
\(
vu=1_{\mathcal B}.
\)
Since $u\in Z(\mathcal B)$,
\(
uv=vu=1_{\mathcal B}.
\)
Therefore $u$ is invertible and $u^{-1}=v$ and the statement (2) is true.

We prove the statement (3) by induction. Let $m\geq 1$.  The case \(m=1\) is trivial,
and the case \(m=2\) is precisely Eq.(3.1). Suppose (3) holds for some \(m\geq 2\). Since $u$ is central,
we have
\[
\begin{aligned}
\Phi(a_1)\cdots\Phi(a_m)\Phi(a_{m+1})
&=
\Phi(a_1\cdots a_m)u^{\,m-1}\Phi(a_{m+1})\\
&=
\Phi(a_1\cdots a_m)\Phi(a_{m+1})u^{\,m-1}\\
&=
\Phi(a_1\cdots a_ma_{m+1})u^m.
\end{aligned}
\]
It follows from the induction that the statement (3) is true.

For (4), define
\(
\Psi(1_{\mathcal A})
=
u^{-1}\Phi(1_{\mathcal A})
=
u^{-1}u
=
1_{\mathcal B}.
\)
Then by Eq.(3.1),  since $u^{-1}\in Z(\mathcal B)$,
\(
\Psi(a)\Psi(b)
=
u^{-1}\Phi(a)u^{-1}\Phi(b)
=
u^{-2}\Phi(a)\Phi(b)
=
u^{-1}\Phi(ab)
=
\Psi(ab)
\)
holds for all $a,b\in\mathcal A$,
That is, $\Psi$ is multiplicative. The bijectivity of $\Psi$ is clear.
\end{proof}

In the sequel, we turn to several special $C^*$-algebras on which the preservers of algebraic maximal numerical range of triple products are precisely central multiples of $*$-isomorphisms.

\begin{Theorem}\label{thm3.3}
Let $\mathcal A$ and $\mathcal B$ be nonzero von Neumann algebras and assume  $\mathcal A$ has no
central summands of type $I_1$. Let $\Phi:\mathcal A\to\mathcal B$ be a surjective map. Then, $\Phi$ satisfies
\[V_0(\Phi(a)\Phi(b)\Phi(c))=V_0(abc)\text{~for all~} a, b, c \in \mathcal{A}.\]
if and only if there exist an element $u\in Z(\mathcal B)$ with $u^3=1_{\mathcal B}$ and a 
$*$-isomorphism $\Psi:\mathcal A\to\mathcal B$ such that $\Phi(a)=u\Psi(a)$ for all $a\in\mathcal A$.
\end{Theorem}

\begin{proof}
For the ``if" part, it is obvious that for all $a,b,c\in\mathcal{A}$ we obtain
\[
\begin{aligned}
V_0(\Phi(a)\Phi(b)\Phi(c))
&= V_0(u\Psi(a) \cdot u\Psi(b) \cdot u\Psi(c)) \\
&= V_0(u^3 \cdot \Psi(a)\Psi(b)\Psi(c)) \\
&= V_0( \Psi(abc)) \\
&= V_0(abc)
\end{aligned}
\]
Therefore, it suffices to show that
\(
V_0\big(\Psi(a)\big) = V_0(a) \text{~holds for all~} a\in\mathcal{A}.
\)
Because $\Psi$ is a $*$-isomorphism, we have $\Psi(1_\mathcal{A}) = 1_\mathcal{B}$, $\|\Psi(a)\| = \|a\|$, $\Psi(a^*a) = \Psi(a)^*\Psi(a)$ for all $a\in\mathcal A$, $\Psi$ maps positive elements of $\mathcal{A}$ into positive elements of $\mathcal{B}$, $\Psi$ is bijective, and $\Psi^{-1}\colon\mathcal{B}\to\mathcal{A}$ is also a $*$-isomorphism.

For any $a\in\mathcal A$ and  arbitrary scalar $\lambda \in V_0(a)$, there exists a state $f \in \mathcal{S}(\mathcal{A})$ such that
\[
f(a) = \lambda ~\text{and}~ f(a^*a) = \|a\|^2.
\]
Define $\hat{f}\colon\mathcal{B}\to\mathbb{C}$ by
\(
\hat{f}(b) = f\big(\Psi^{-1}(b)\big)\) for \( b\in\mathcal{B}.
\)
Since $\Psi^{-1}$ preserves positivity and $f$ is a positive linear functional, the map $\hat{f} = f\circ\Psi^{-1}$ is also a positive linear functional. As
\(
\hat{f}(1_\mathcal{B}) = f\big(\Psi^{-1}(1_\mathcal{B})\big) = f(1_\mathcal{A}) = 1,
\)
we see that $\hat{f} \in \mathcal{S}(\mathcal{B})$.
Furthermore,
\(
\hat{f}\big(\Psi(a)\big) = f\big(\Psi^{-1}(\Psi(a))\big) = f(a) = \lambda.
\)
Using the $*$-multiplicativity and isometry of $\Psi$,
\[
\begin{aligned}
\hat{f}\big(\Psi(a)^*\Psi(a)\big)
= \hat{f}\big(\Psi(a^*a)\big)
= f\big(\Psi^{-1}(\Psi(a^*a))\big)
= f(a^*a)
= \|a\|^2
= \|\Psi(a)\|^2.
\end{aligned}
\]
This ensures that $\lambda \in V_0\big(\Psi(a)\big)$, and hence $V_0(a) \subseteq V_0\big(\Psi(a)\big)$.
The reverse inclusion $V_0\big(\Psi(a)\big) \subseteq V_0(a)$ follows from exactly the same argument applied to the  $*$-isomorphism $\Psi^{-1}$.
Hence we conclude that $V_0(a) = V_0\big(\Psi(a)\big)$. So the condition is sufficient.

We will prove the ``only if" part by a series of Claims. Since von Neumann algebras are unital \(C^*\)-algebras, Theorem \ref{thm3.1} applies. So we have $\Phi=u\Psi$ with $\Psi$ a unital bijective multiplicative map and $u=\Phi(I_{\mathcal A})\in Z(\mathcal B)$  invertible.

\textbf{Claim 1.} $\Psi(Z(\mathcal A))=Z(\mathcal B)$.

Let $z\in Z(\mathcal{A})$ and $y\in\mathcal{B}$. By surjectivity,
there exists $x\in\mathcal{A}$ such that $y=\Psi(x)$. Using the
multiplicativity of $\Psi$, we obtain
\[\Psi(z)y=\Psi(z)\Psi(x)=\Psi(zx)=\Psi(xz)=\Psi(x)\Psi(z)=y\Psi(z).\]
Thus, $\Psi(z)\in Z(\mathcal{B})$, and consequently
\(
\Psi\bigl(Z(\mathcal{A})\bigr)\subseteq Z(\mathcal{B}).
\)

Since the inverse of a bijective multiplicative map is multiplicative,
the same argument applied to $\Psi^{-1}$ gives
\(
\Psi^{-1}\bigl(Z(\mathcal{B})\bigr)\subseteq Z(\mathcal{A}).
\)
Therefore,
\(
\Psi\bigl(Z(\mathcal{A})\bigr)=Z(\mathcal{B}).
\)

\textbf{Claim 2.} $\Psi(\mathcal P(Z(\mathcal A)))=\mathcal P(Z(\mathcal B))$.

Let \(p\in\mathcal{P}(Z(\mathcal{A}))\). By Claim~1,
\(\Psi(p)\in Z(\mathcal{B})\), and
\(
\Psi(p)^2=\Psi(p^2)=\Psi(p).
\)
Hence, \(\Psi(p)\) is an idempotent in the commutative \(C^*\)-algebra
\(Z(\mathcal{B})\). As every idempotent in a commutative \(C^*\)-algebra
is a projection, we see that
\(
\Psi(p)\in\mathcal{P}(Z(\mathcal{B})).
\)
Therefore,
\(
\Psi\bigl(\mathcal{P}(Z(\mathcal{A}))\bigr)
\subseteq
\mathcal{P}(Z(\mathcal{B})).
\)
\if false
Since \(\Psi^{-1}\) is also a bijective multiplicative map and, by
Claim 1,
\(
\Psi^{-1}\bigl(Z(\mathcal{B})\bigr)=Z(\mathcal{A}),
\)\fi
The same argument applied to \(\Psi^{-1}\) yields the reverse
inclusion. Thus,
\(
\Psi\bigl(\mathcal{P}(Z(\mathcal{A}))\bigr)
=
\mathcal{P}(Z(\mathcal{B})).
\)

\textbf{Claim 3.} $u^3=1_{\mathcal B}\in Z(\mathcal B)$ and hence $V_0(\Psi(a))=V_0(a)$ holds for all $a\in\mathcal A$.

Put $v=u^3$. Since $u\in Z(\mathcal B)$, we have $v\in Z(\mathcal B)$. We first show that
\(
V_0(vq)=\{1\}
\) for every nonzero central projection $q\in Z(\mathcal B)$.

Let \(q\in\mathcal P(Z(\mathcal B))\) be nonzero. By Claim 2, there exists
\(p\in\mathcal P(Z(\mathcal A))\) such that \(\Psi(p)=q\). Since \(\Psi\) is injective and \(\Psi(0)=0\), we
have \(p\neq0\). From $\Phi=u\Psi$ and $V_0(\Phi(a)\Phi(b)\Phi(c))=V_0(abc)$, we obtain
\[
V_0(v\Psi(a)\Psi(b)\Psi(c))=V_0(abc)
\]
for all \(a,b,c\in\mathcal A\). Taking
\(
a=b=c=p,
\)
we obtain $V_0(vq)=V_0(p)$. Since \(p\) is a nonzero projection,
\(\|p\|=1\), and every \(f\in \mathcal{S}_{\max}(p)\) satisfies \(f(p)=f(p^*p)=1\). Hence \(V_0(p)=\{1\}\), and therefore \(V_0(vq)=\{1\}\).

We now prove that $v=1_{\mathcal B}$. Suppose, to the contrary, that $v\neq 1_{\mathcal B}$. Then there exists $\lambda\in\sigma(v)$ with $\lambda\neq1$.
Furthermore, \(v=u^3\) is invertible, so $0\notin\sigma(v)$. Choose $r>0$ such that $0<r<\min\{|\lambda|, |\lambda-1|\}$.
Let
\(
\Delta:=\overline{D(\lambda,r)}
\)
and define
\(
q:=E_v(\Delta),
\)
 where \(E_v\) denotes the spectral measure of \(v\).
Since
\(\lambda\in\sigma(v)\), the spectral projection \(q\) is nonzero.
Moreover, because \(v\in Z(\mathcal B)\), all its spectral projections belong to \(Z(\mathcal B)\). Thus,
\(
q\in\mathcal P(Z(\mathcal B)).
\)

By the functional calculus, $\|(v-\lambda 1_{\mathcal B})q\|\leq r$. In fact,
because of $v\in Z(\mathcal B)$, it is normal. By the spectral theorem, there exists a spectral measure $E_v$ such that
\[
v=\int_{\sigma(v)} z \mathrm d E_v(z).
\]
\if false where $z$ is the inclusion function on $\sigma(v)$.\fi
Then
\[
(v-\lambda 1_{\mathcal B})q = \left(\int_{\sigma(v)} (z-\lambda) \mathrm d E_v(z)\right)E_v(\Delta) = \int_{\Delta\cap\sigma(v)} (z-\lambda) \mathrm d E_v(z).
\]
Note that
\[
\left\|\int_{\Delta\cap\sigma(v)} (z-\lambda) \mathrm d E_v(z)\right\| \leq \sup_{z\in\Delta\cap\sigma(v)} |z-\lambda|\leq \sup_{z\in\Delta} |z-\lambda| \leq r.
\]
So $\|(v-\lambda 1_{\mathcal B})q\| \leq r$.

Also, $\|vq\|\neq 0$. Indeed, if $\|vq\|=0$, then $vq=0$, so $q=v^{-1}vq=0$, contrary to the choice of \(q\). Hence $\|vq\|>0$.

Let \(\mu\in V_{0}(vq)\). Then there exists
\(
f\in \mathcal{S}_{\max}(vq)
\)
such that
\(
\mu=f(vq)
\)
and
\(
f\bigl((vq)^*(vq)\bigr)=\|vq\|^2.
\)
 Since $vq=vq^2=qvq$, we have $vq\in q\mathcal B q$. Hence,
\(
0\leq (vq)^*vq \leq \|vq\|^2q.
\)
Applying \(f\) gives
\(
\|vq\|^{2}=f((vq)^*vq) \leq f(\|vq\|^2q)= \|vq\|^{2}f(q).
\)
Because \(\|vq\|>0\), it follows that
\(
f(q)=1.
\)
Then
\[
|\mu-\lambda|=|f((v-\lambda 1_{\mathcal B})q)| \leq \|(v-\lambda 1_{\mathcal B})q\| \leq r,
\]
Hence $V_0(vq)\subseteq \overline{D(\lambda,r)}$. But $r<|\lambda-1|$ implies $1\notin \overline{D(\lambda,r)}$, and consequently $1\notin V_0(vq)$, contradicting $V_0(vq)=\{1\}$.

Therefore, we must have $v=1_{\mathcal B}$, that is, $u^3=1_{\mathcal B}$. Now it is clear that
$$ V_0(\Psi(a))=V_0(\Phi(a)\Phi(1_{\mathcal A})\Phi(1_{\mathcal A}))=V_0(a)
$$
for every $a\in\mathcal A$.

\textbf{Claim 4.} $\Psi$ is linear.

By Claim 3, $V_0(\Psi(a))=V_0(a)$ for every $a\in\mathcal A$.
Let
$\lambda\in\mathbb C$, $a\in\mathcal A$, and $x\in\mathcal B$.
By the surjectivity of $\Psi$, there exists $y\in\mathcal A$ such that
\(
\Psi(y)=x.
\)
Then we get
\begin{align*}
V_0(\Psi(\lambda a)x) &= V_0(\Psi(\lambda a)\Psi(y))
= V_0(\Psi(\lambda a y))
= V_0(\lambda a y)
= \lambda V_0(a y)\\
&= \lambda V_0(\Psi(a y))
= \lambda V_0(\Psi(a)\Psi(y))
= V_0(\lambda \Psi(a)\Psi(y))
= V_0(\lambda \Psi(a)x).
\end{align*}
By Proposition \ref{prop2.3}, we conclude
\(
\Psi(\lambda a) = \lambda \Psi(a)
\) for all $\lambda\in\mathbb C$ and $a\in\mathcal A$.  Hence, $\Psi$ is homogeneous.
 As the von Neumann algebra $\mathcal A$ has no central summand of type $I_1$, according to \cite[Theorem 1]{song2014characterizing}, we conclude that $\Psi$ is additive.

\textbf{Claim 5.}
If $p\in\mathcal{A}$ is a projection, then $\Psi(p)\in\mathcal{B}$ is a projection.

If $p = 0$ or $p = 1_\mathcal{A}$, then $\Psi(0) = 0$ and $\Psi(1_\mathcal{A}) = 1_\mathcal{B}$. Both $0$ and the identity are projections in $\mathcal{B}$, so the assertion holds trivially.

Assume $p\in\mathcal P(\mathcal A)$ is nontrivial. Then
\(\sigma(p)=\{0,1\}\). Put \(s=1-2p\). Then
\(s=s^*\) and \(s^2=1\), and hence
\[
\sigma(s)=\sigma(1-2p)=\{-1,1\}
\]
and \(\|s\|=1\). Since \(s\) is
self-adjoint, according to Proposition \ref{prop2.6}, we have
\(
V_0(s)=[-1,1],
\)
and hence,
\(
V_0(\Psi(s))=V_0(s)=[-1,1].
\)
Moreover, as
\(
\Psi(s)^2=\Psi(s^2)=\Psi(1_\mathcal{A}) = 1_\mathcal{B}.
\)
Applying Proposition \ref{prop2.5}, we get
\(
\Psi(s)=\Psi(s)^*.
\)
It follows that
\[
\Psi(p)=\frac{1_\mathcal{B}-\Psi(s)}{2}
\]
is self-adjoint. On the other hand,
\(
\Psi(p)^2=\Psi(p^2)=\Psi(p).
\)
Hence \(\Psi(p)\) is a projection.

\textbf{Claim 6.} $\Psi$ is continuous.

Since $\Psi$ is linear and both $\mathcal{A}$ and $\mathcal{B}$ are Banach spaces, the Closed Graph Theorem shows that it suffices to prove that the graph of $\Psi$ is closed.
Equivalently, it suffices  to verify that if $(a_n)_{n\geq1} \subseteq  \mathcal{A}$ satisfies $\| a_n\|  \to 0$ and $\| \Psi(a_n)- b\| \to 0$, then $b = 0$.

By the surjectivity of $\Psi$, there exists \(c\in \mathcal{A}\) such that
\(\Psi(c)=b^*\). For each \(n\), choose \(\lambda_n\in V_0(a_nc)\).
By the assumption that $\Psi$ preserves the maximal numerical range, we have
\[
V_0(a_n c) = V_0\bigl(\Psi(a_n c)\bigr) = V_0\bigl(\Psi(a_n)\Psi(c)\bigr) = V_0\bigl(\Psi(a_n) b^*\bigr).
\]
Hence $\lambda_n \in V_0\bigl(\Psi(a_n) b^*\bigr)$.
On the other hand,
\(
V_0(x) \subseteq V(x) \subseteq \bigl\{ z\in\mathbb{C} : |z| \le \|x\| \bigr\}
\)
for every
element \(x\) of a unital \(C^*\)-algebra. Therefore,
\[
|\lambda_n|\leq\|a_nc\|\leq\|a_n\|\,\|c\|\to0.
\]
Moreover, \(\| \Psi(a_n)b^* - bb^*\| \to 0\).
Now,
it follows from Proposition \ref{prop2.7} that
\[
0 \in V_0(b b^*).
\]
By Proposition \ref{prop2.6} and the positivity of $bb^*$, we have
\[
0\in V_0(bb^*)=\{\|bb^*\|\}=\{\|b\|^2\},
\]
which forces \(b=0\).
Thus $\Psi$ is a closed operator. By the closed graph theorem,
\(\Psi\) is continuous.

\textbf{Claim 7.} $\Psi$ is a $*$-isomorphism.

By Claim 4, $\Psi$ is an algebraic isomorphism. So we only need to check that $\Psi(a^*)=\Psi(a)^*$ holds for all $a\in\mathcal{A}$.

We first prove that \(\Psi\) maps self-adjoint elements to self-adjoint elements.
Let $a = a^* \in \mathcal{A}$. Since \(\mathcal{A}\) is a von Neumann algebra, there exists real linear combination of projections
\[
a_n=\sum_{k=1}^{m_n}\lambda_{n,k}p_{n,k}.
\]
in $\mathcal A$ such that
\(\| a_n-a\| \to 0\).
Since \(\Psi\) is linear and maps projections to projections by Claims 4 and 5, we have
\[
\Psi(a_n)=\sum_{k=1}^{m_n}\lambda_{n,k}\Psi(p_{n,k}).
\]
is self-adjoint for every
\(n\). Then the continuity of \(\Psi\) (Claim 6) leads to
\(\| \Psi(a_n)-\Psi(a)\| \to 0\). Therefore, as a limit of self-adjoint elements, \(\Psi(a)=\Psi(a)^*\), that is, $\Psi$ sends self-adjoint elements into self-adjoint elements.

Now \if false let $x\in\mathcal{A}$ be arbitrary. Write $x = h + ik$, where
\[
h = \frac{x + x^*}{2}, ~ k = \frac{x - x^*}{2i}
\]
are both self-adjoint elements of $\mathcal{A}$. Then
\(
\Psi(x^*) = \Psi(h - ik) = \Psi(h) - i\Psi(k).
\)
On the other hand,
\[
\Psi(x)^* = \bigl(\Psi(h) + i\Psi(k)\bigr)^* = \Psi(h)^* - i\Psi(k)^* = \Psi(h) - i\Psi(k).
\]
Therefore\fi it is obvious that $\Psi(x^*) = \Psi(x)^*$ for all $x\in\mathcal{A}$ and thus $\Psi$ is a $*$-isomorphism.
\end{proof}

For C$^*$-algebras, we have

\begin{Theorem}\label{thm3.4}
Let $\mathcal{A}$ and $\mathcal{B}$ be  unital prime
$C^*$-algebras. Assume $\mathcal{A}$ is of real rank zero and   $\Phi:\mathcal A\to \mathcal B$ is a surjective map. Then $\Phi$ satisfies
\[
V_0(\Phi(a)\Phi(b)\Phi(c))=V_0(abc) \text{~for all~} a, b, c \in \mathcal{A}
\]
if and only if there exists a scalar $\varepsilon \in\mathbb C$ with $\varepsilon ^3=1$ and a 
$*$-isomorphism $\Psi:\mathcal A\to\mathcal B$ such that
\(
    \Phi=\varepsilon \Psi.
\)
\end{Theorem}

\begin{proof}
The ``if" part is obvious. For the "only if" part, we only present the steps that differ from the proof of Theorem \ref{thm3.4}.

Since $\mathcal B$ is a prime $C^*$-algebra, we have
\(
    Z(\mathcal B)=\mathbb C 1_{\mathcal B}
\)
(Ref. \cite[Proposition 2.6]{omland2014primeness}).
According to Claim 2 in the proof of Theorem \ref{thm3.4}, there exists a nonzero scalar $\varepsilon\in\mathbb C$ such that
\(
    \Phi(1_{\mathcal A})=\varepsilon 1_{\mathcal B}.
\)
Taking $a=b=c=1_{\mathcal A}$ in the hypothesis gives
\[
    V_0(\varepsilon^3 1_{\mathcal B})=V_0\bigl(\Phi(1_{\mathcal A})^3\bigr)=V_0(1_{\mathcal A}).
\]
Hence
\(
    \varepsilon^3=1.
\)
Let $\Psi=\varepsilon^{-1}\Phi$. $\Psi$ is bijective, homogeneous
 and multiplicative.
\(
    \Phi(a)=\varepsilon\Psi(a) \text{~for all a $\in \mathcal{A}$~}.
\)

By the hypothetical conditions,
 $\mathcal A$ is  prime as a ring. Assume that $\mathcal A  \neq \mathbb{C}1_\mathcal A$, then it contains a nontrivial projection because it is of real rank zero.  By \cite{martindale1969when}, it follows that
$\Psi$ is additive and hence an algebraic isomorphism.
Furthermore, by applying Proposition \ref{prop2.5} and a similar argument as in Claims 5-6 in the proof of Theorem \ref{thm3.4},  $\Psi$ is continuous and sends projections into projections. Since $\mathcal A$ is of real rank zero, every self-adjoint element is the limit of finite real linear combination of projections. This implies that $\Psi$ sends self-adjoint elements into self-adjoint elements. So, $\Psi$ is a $*$-isomorphism.

If
\(
\dim\mathcal{A}=1,
\)
we have
\(
\mathcal{B}=\mathbb{C}1_{\mathcal{B}},
\)
and $\Psi(a)\Psi(x)=\Psi(ax)=a\Psi(x)$. Therefore, $\Psi(\lambda 1_{\mathcal A})=\lambda 1_{\mathcal B}$ for all $\lambda\in\mathbb C$, which means that $\Psi$ is still a $*$-isomorphism.
\end{proof}

Particularly, we have

\begin{Corollary}
Let $\mathcal A$ and $\mathcal B$ be  factor von Neumann algebras,
and let $\Phi:\mathcal A\to \mathcal B$ be a surjective map. Then $\Phi$ satisfies
\[
V_0(\Phi(a)\Phi(b)\Phi(c))=V_0(abc) \text{~for all~} a, b, c \in \mathcal{A}
\]
if and only if there exist a scalar $\varepsilon \in\mathbb C$ with $\varepsilon ^3=1$ and a 
$*$-isomorphism $\Psi:\mathcal A\to\mathcal B$ such that
\(
    \Phi=\varepsilon \Psi.
\)
\end{Corollary}


\end{document}